\documentclass[a4paper,10pt]{article}

\title{Density of Primes in $l^{th}$ Power Residues}
\author{R. Balasubramanian, Prem Prakash Pandey}
\usepackage{amsmath}
\usepackage{amssymb}
\usepackage{amsthm}
\usepackage{textcomp}
\usepackage{mathtools}
\newtheorem{lem}{Lemma}
\newtheorem{rem}{Remark}
\newtheorem{thm}{Theorem}

\begin{document}

\maketitle

\begin{abstract}
Given a prime number $l$, a finite set of integers $S=\{a_1, \ldots ,a_m\}$  and $m$ many $l^{th}$ roots of unity $\zeta_l^{r_i}, i=1, \ldots ,m$ we study the distribution of primes $\mathbf{p}$ in $\mathbb{Q}(\zeta_l)$ such that the $l^{th}$ residue symbol of $a_i$ with respect to $\mathbf{p}$ is $\zeta_l^{r_i}, \mbox{ for all } i$. We find out that this is related to the degree of the extension $\mathbb{Q}(a_1^{\frac{1}{l}}, \ldots ,a_m^{\frac{1}{l}})/\mathbb{Q}$. We give an algorithm to compute this degree. Also we relate this degree to rank of a matrix obtained from $S=\{a_1, \ldots ,a_m\}$. This latter argument enables one to describe the degree $\mathbb{Q}(a_1^{\frac{1}{l}}, \ldots ,a_m^{\frac{1}{l}})/\mathbb{Q}$ in much simpler terms.
\end{abstract}

\section{Introduction}
In a recent paper \cite{RB} the authors have computed the relative density of primes for which a given finite string $S=\{a_1, \ldots ,a_m\}$ of integers are quadratic residues simultaneously. It turns out, via Chebotarev density theorem, that this density is reciprocal of the degree of the multiquadratic extension given by square roots of the finite string of given integers over $\mathbb{Q}$. Given a field $K$ which contains an $n^{th}$ root of unity and given a finite set of integers $S=\{a_1, \ldots ,a_m\}$ one can determine the degree of the extension $K(a_1^{\frac{1}{n}}, \ldots ,a_m^{\frac{1}{n}})/K$ using Galois theory, for instance see \cite{SHW}. In this article, we  study the distribution of primes $\mathbf{p} \mbox{ of } \mathbb{Q}(\zeta_l)$ modulo which each of $a_i$ assumes a preassigned $l^{th}$ power residue symbol and then relate it to the degree of the extension $\mathbb{Q}(a_1^{\frac{1}{l}}, \ldots ,a_m^{\frac{1}{l}})/\mathbb{Q}$. We give two methods to compute the degree of the extension $\mathbb{Q}(a_1^{\frac{1}{l}}, \ldots ,a_m^{\frac{1}{l}})/\mathbb{Q}$ and either of the two methods may prove to be useful at a given instance.

In section 2, we use a ramification argument in place of the classical use of the Eisenstein criterion to compute the degree of the extension $\mathbb{Q}(a_1^{\frac{1}{l}}, \ldots ,a_m^{\frac{1}{l}})/\mathbb{Q}$. Section 3 deals with the $l^{th}$ power residue symbols and study of the distribution of primes $\mathbf{p}$ modulo which they take a fixed value for each $a_i$. In section 4, we define a matrix $T$ and then proceed to relate the degree of the extension to the rank of $T$. The tools we use are basic in nature and can be worked out from any basic text on algebraic number theory but for the sake of completion we will give some of the proofs.\\

We will fix an odd prime $l$ and $\zeta_l$ will stand for a fixed primitive $l^{th}$ root of unity.

\section{Determination Of The Degree}

To start with, we can assume that $a_i^{'s}$ are $l^{th}$ power free and none of them is $1$.\\

One has following;
\begin{lem} 
If $a\in \mathbb{Z}$ is not a $l^{th}$ power then $X^l-a$ is irreducible over $\mathbb{Z}$.
\end{lem}

\begin{lem}
 Let $b_1, b_2,  \ldots , b_i$ be some integers and $b$ be an intger which is not a $l^{th}$ power and such that there is a prime $q \neq l$ which divides $b$ but does not divide any of $b_j$. Then $[\mathbb{Q}(b_1^{\frac{1}{l}}, \ldots ,b_i^{\frac{1}{l}},b^{\frac{1}{l}}):\mathbb{Q}(b_1^{\frac{1}{l}}, \ldots ,b_i^{\frac{1}{l}})]=l$.
\end{lem}

\begin{proof}
 Let us write $L=\mathbb{Q}(b_1^{\frac{1}{l}}, \ldots ,b_i^{\frac{1}{l}},b^{\frac{1}{l}})$ and $K=\mathbb{Q}(b_1^{\frac{1}{l}}, \ldots ,b_i^{\frac{1}{l}})$. \\
Since $q \mbox{ does not divide } b_j \mbox{ for any }j$, $q$ is unramified in each of $\mathbb{Q}(b_j^{\frac{1}{l}})$ and hence it is unramified in $K$ as well. On the other hand looking at complex factorization of $X^l-b$ we find that if $X^l-b=f_1(X)f_2(X)$ in $K[X]$ then $f_1(0)=b^{\frac{r}{l}}\zeta_l^c \in K$ for some integers $r \mbox{ and } c$. This will force that $q$ ramifies in $K$, so we have that the polynomial $X^l-b$ is irreducible in $K[X]$. This proves the lemma.
\end{proof}

\underline{Algorithm to compute the degree $[\mathbb{Q}(a_1^{\frac{1}{l}}, \ldots ,a_m^{\frac{1}{l}}):\mathbb{Q}]$} \\
Claim: We can obtain some integers $b_1, \ldots ,b_t$ with following properties;\\
(1) There is a prime number $q_i|b_i$ which does not divide $b_j \mbox{ for }j\neq i$.\\
(2) None of $b_i$ is a $l^{th}$ power.\\
(3) $\mathbb{Q}(a_1^{\frac{1}{l}}, \ldots ,a_m^{\frac{1}{l}})=\mathbb{Q}(b_1^{\frac{1}{l}}, \ldots ,b_t^{\frac{1}{l}})$.\\
We will generate the numbers $b_i's$ in successive steps. Here upper index will indicate the number of steps.\\
Let $q_1$ be a prime divisor of $a_1$ we put $b_1^{(1)}=a_1$. For $i>1$ if $q_1\nmid a_i$ then we will put $b_i^{(1)}=a_i$ and in case $q_1|a_i$ then we will define $b_i^{(1)}$ as follows:\\
Let $r_1 \mbox{ and }r_i$ be the exponent of $q_1 \mbox{ in } a_1 \mbox{ and }a_i$ respectively. Without loss of generallity we can assume that $1 \leq r_1,r_i \leq l-1$. As $m_i$ runs modulo $l$ the numbers $m_ir_1+r_i$ are distinct modulo $l$ hence for some choice of $m_i$ we will have $m_ir_1+r_i=\lambda_i l$ and then we define $b_i^{(1)}=\frac{a_1^{m_i}a_i}{q_1^{\lambda_i l}}$. If any of $b_i^{(}1)$ happens to be a $l^{th}$ power then we will omit it and consider only those $b_i^{(1)}$ which are not $l^{th}$ powers, say, $b_1^{(1)}, \ldots ,b_{s_1}^{(1)}$. Now one has $q_1|b_1^{(1)} \mbox{ and } q_1\nmid b_i^{(1)} \mbox{ for } i>1$, and $\mathbb{Q}(a_1^{\frac{1}{l}}, \ldots ,a_m^{\frac{1}{l}})=\mathbb{Q}((b_1^{(1)})^{\frac{1}{l}}, \ldots ,(b_{s_1}^{(1)})^{\frac{1}{l}})$.
Next we set $b_1^{(2)}=b_1^{(1)}, b_2^{(2)}=b_2^{(1)}$ and start with a prime divisor $q_2 \mbox{ of } b_2^{(2)}$ and repeat the same process to obtain $b_i^{(2)}$. Suppose this process stops at $k^{th}$ step then $b_1^{(k)}, \ldots b_{s_t}^{(k)}$ are the required numbers. We will put $t=s_t$ and if $p=q_i \mbox{ for some }i$ then we call $b_1=b_i^{(k)}$ and rest $t-1$ numbers can be taken in any  order and if $p\neq q_i \mbox{ for some }i$ then we can take any ordering.
Now Lemma 1 gives that $[\mathbb{Q}(b_1^{\frac{1}{l}}):\mathbb{Q}]=l$. Then we use Lemma 2 successively to obtain $[\mathbb{Q}(b_1^{\frac{1}{l}}, \ldots ,b_t^{\frac{1}{l}}):\mathbb{Q}]=l^t$, which is the required degree.

\section{$l^{th}$ power residue symbol}
Let $p$ be a prime different from $l$ and $f$ be the inertia degree of $p$ in $\mathbb{Q}(\zeta_l)$. For any prime ideal $\mathbf{p}$ of $\mathbb{Q}(\zeta_l)$ dividing $p$ and an integer $\alpha \in \mathbb{Q}(\zeta_l)$ not contained in $\mathbf{p}$ one has $l|{p^f-1} \mbox{ and } \alpha^{p^f-1}\equiv 1 \mbox{ (mod }\mathbf{p})$. Hence there is an $l^{th}$ root of unity $\zeta_l^i, 0<i \leq l$ such that $\alpha^{\frac{p^f-1}{l}}\equiv \zeta_l^i \mbox{ (mod }\mathbf{p})$. Since $l^{th}$ roots of unity are distinct modulo $\mathbf{p}$, there is a unique such $i$ and we will define $(\frac{\alpha}{\mathbf{p}})_l =\zeta_l^i$.
We state some results on the higher residue symbols \cite{DH,EM}.\\

\begin{thm}(Kummer's Criterion)
 $(\frac{\alpha}{\mathbf{p}})_l \equiv  \alpha^{\frac{p^f-1}{l}} \mbox{ (mod }\mathbf{p})$.
\end{thm}

\begin{thm}
 The $l^{th}$ power residue symbols are completely multiplicative.
\end{thm}

\begin{thm}
 $\alpha \in \mathbb{Q}(\zeta_l)$ is an $l^{th}$ power modulo $\mathbf{p}$ iff $(\frac{\alpha}{\mathbf{p}})_l =1$.
\end{thm}

Given any ideal $\mathbb{a}$ of $\mathbb{Q}(\zeta_l)$, we will define  $(\frac{\alpha}{\mathbb{a}})_l=\prod_{\mathbf{p}| \mathbb{a}} (\frac{\alpha}{\mathbf{p}})_l$ with multiplicity counted. For $\beta \in \mathbb{Q}(\zeta_l)$ we will define $(\frac{\alpha}{\beta})_l=(\frac{\alpha}{(\beta)})_l$ where $(\beta)$ stands for the principal ideal generated by $\beta$.\\
An integer $\alpha \in \mathbb{Q}(\zeta_l)$ is called primary if it is congruent to a rational integer modulo $(1-\zeta_l)^2$.

\begin{thm}(Eisenstein's Reciprocity Law)
 If $\alpha$ is a primary integer and $a$ is a rational integer both coprime to each other and coprime to $l$ then one has $(\frac{a}{\alpha})_l=(\frac{\alpha}{a})_l$.
\end{thm}

We will now introduce some more terminologies and see an alternate way to define the $l^{th}$ residue symbol.\\
Let $L/K$ be a Galois extension of number fields. For an unramified prime $\wp$ of $L$ we will write $k_L$ for the residue field of $\wp$ and $k$ will denote the residue field of $\wp\cap \mathbb{O}_K$. There is an exact sequence \cite{JN},\\
\begin{equation*}
 0\longrightarrow D_{\wp} \longrightarrow Gal(k_L/k) \longrightarrow 0,
\end{equation*}
where  $D_{\wp}$ is the decomposition group of $\wp$. Let $\sigma_{\wp} \in Gal(k_L/k)$ be the Frobenius at $\wp$ then its inverse image (in the above exact sequence) in $D_{\wp}$ is called Artin symbol of $\wp$ for the extension $L/K$ and is written as $\left( \frac{\wp}{L/K} \right) $.\\
We note that if $\wp \mbox{ and } \wp^{'}$ are primes in $L$ above the same prime of $K$ then $\left( \frac{\wp}{L/K} \right)$ $\mbox{ and } \left( \frac{\wp^{'}}{L/K} \right) $ are Galois conjugate by an element of $Gal(L/K)$ which permutes $\wp$ and $\wp^{'}$. In particular if $L/K$ is abelian then for prime $ \mathbf{p}=\wp\cap \mathbb{O}_K$ we can define $\left( \frac{\mathbf{p}}{L/K} \right)=\left( \frac{\wp}{L/K} \right) $.\\

For any $\sigma \in Gal(L/K)$ let $P_{L/K}(\sigma)$ denote the set of prime ideals $\mathbf{p}$ in $K$ such that there is a prime ideal $\wp$ of $L$ above $\mathbf{p}$ such that $\left( \frac{\wp}{L/K} \right) = \sigma $. \\
\begin{thm} (Cebotarev Density Theorem)
 Let $\sigma \in Gal(L/K)$ and $C_{\sigma}$ stand for conjugacy class of $\sigma$. Then density of $P_{L/K}(\sigma)$ is $\frac{|C_{\sigma}|}{[L:K]}$.
\end{thm}

In particular, if $L/K$ is abelian then the density of primes $\mathbf{p}$ of $K$ such that $\left( \frac{\mathbf{p}}{L/K} \right)=\sigma $ is $\frac{1}{[L:K]}$.\\

We have following,
\begin{lem}
 Consider a Galois extension $L/K$ of number fields and let $F$ be an intermediate field such that $F/K$ is Galois. Then for any unramified prime $\wp \mbox{ of }L$ one has $\left( \frac{\wp}{L/K} \right)_{|F}=\left( \frac{\wp \cap \mathbb{O}_F}{L/K} \right)$.
\end{lem}

For a fixed prime $l$, we want to define the $l^{th}$ residue symbol $(\frac{a}{p})_l$ for each prime $p \neq l$ and integer $a$ coprime to $p$. We will write $f_a(X)=X^l-a$ and let $K_a$ denote splitting field of $f_a(X)$. Then $K_a\supset \mathbb{Q}(\zeta_l)$. For any prime $\mathbf{p}$ of $\mathbb{Q}(\zeta_l)$ above $p$ we let $\sigma_{\mathbf{p},a} \in Gal(K_a/ \mathbb{Q}(\zeta_l))$ denote the Artin symbol for the prime $\mathbf{p}$ for the extension $K_a/ \mathbb{Q}(\zeta_l)$. Then $\sigma_{\mathbf{p},a} (a^{1/l})/a^{1/l}$ is an $l^{th}$ root of unity and we will define $\left( \frac{a}{\mathbf{p}} \right)_l=\sigma_{\mathbf{p},a} (a^{1/l})/a^{1/l}$. \\

\begin{lem}
 The definition of $l^{th}$ residue symbol given here and the one given earlier are equivalent.
\end{lem}

\begin{proof}

Let $f$ be the inertia degree of $\mathbf{p}$ in the extension $\mathbb{Q}(\zeta_l)/ \mathbb{Q}$. Then the Artin symbol $\sigma_{\mathbf{p},a}$ satisfies
\begin{equation*}
 \sigma_{\mathbf{p},a} (a^{\frac{1}{l}}) \equiv a^{\frac{p^f}{l}} \quad ( \mbox{ mod }\mathbf{p} ),
\end{equation*}
\begin{equation*}
 i.e. \quad \frac{\sigma_{\mathbf{p},a} (a^{\frac{1}{l}})}{a^{\frac{1}{l}}} \equiv a^{\frac{p^f-1}{l}} \quad ( \mbox{ mod }\mathbf{p} ),
\end{equation*}
which proves the lemma.

\end{proof}

The follwoing theorem is well known but we supply a proof for the sake of completeness.
\begin{thm}
 If $n$ is not an $l^{th}$ power of an integer then the estimate $$\sum_{\mathbf{p}; Norm(\mathbf{p}) \leq x}(\frac{n}{\mathbf{p}})_l = o(\pi (x))$$ holds as $x\rightarrow \infty$. Here Norm denotes for Norm map of the extension $\mathbb{Q}(\zeta_l)/ \mathbb{Q}$.
\end{thm}

\begin{proof}
 One has $$\sum_{\mathbf{p}; Norm(\mathbf{p}) \leq x}(\frac{n}{\mathbf{p}})_l = {\sum_{b=1}^l}\left( \sum_{Norm(\mathbf{p}) \leq x, \sigma_{\mathbf{p},n}=\tau_b}(\frac{n}{\mathbf{p}})_l\right),$$ where $\tau_b$ is automorphism of $K_n$ which sends $n^{\frac{1}{l}}\longmapsto (\zeta_l)^b n^{\frac{1}{l}}$. hence \\
$$ \sum_{\mathbf{p}; Norm(\mathbf{p}) \leq x}(\frac{n}{\mathbf{p}})_l= \sum_{b=1}^l \left( \sum_{Norm(\mathbf{p}) \leq x, \sigma_{\mathbf{p},n}=\tau_b} \zeta_l^b \right).$$\\  Thus we obtain
$$\sum_{\mathbf{p}; Norm(\mathbf{p}) \leq x}(\frac{n}{\mathbf{p}})_l= {\sum_{b=1}^{l}} \zeta_l^b (\frac{1}{l} \pi (x)+o(\pi (x)))=\frac{1}{l} \pi(x) {\sum_{b=1}^{l}} \zeta_l^b+o(\pi(x))=o(\pi(x)).$$ Because of the nontriviality of the character, the first term is zero. This proves the result.
\end{proof}

Given  $m$ elements $\zeta_l^{r_i}$ in $\mu_l$, the group of $l^{th}$ roots of unity, not necessarily distinct, we want to determine density of primes $\mathbf{p}$ which satisfy $(\frac{a_i}{\mathbf{p}})_l=\zeta_l^{r_i}$. For this, we will consider the counting function $$S_x=\frac{1}{ul^m} \sum_{\mathbf{p}; Norm(\mathbf{p}) \leq x, \mathbf{p} \notin S} \prod_{k=1}^m\left( \prod_{j=1, j\neq r_k}^l(\zeta_l^j-(\frac{a_k}{\mathbf{p}})_l)\right),$$ here $S$ is the set of primes dividing $a_1 \ldots a_m$ and $u$ is a unit satisfying $$ul^m=\prod_{k=1}^m\prod_{j=1, j\neq r_k}^l(\zeta_l^j-\zeta_l^{r_k}).$$ We note that $S_x$ exactly counts number of primes $\mathbf{p}$ of Norm up to $x$ which satisfy $(\frac{a_i}{\mathbf{p}})_l=\zeta_l^{r_i}$ for all $i$.
 Note that the choices of $r_i$ can not be arbitrary because of the multiplicativity of $l^{th}$ power residue symbol. i.e. to say that the assignment $a_i\longrightarrow \zeta_l^{r_i}$ shall be restriction of some morphism of semigroups $\mathbb{Z}^*/{{\mathbb{Z}^*}^l}\longrightarrow \mu_l$, but the counting function already takes care of this. To show this we note that any multiplicative relation among $a_i^{'}$s can be brought into the form $\prod_{k=1}^m a_i^{c_i}=c^l$ for some integers $c_i \mbox{ and }c$. Now the corresponding relation expected in $\mu_l$ is $\prod_{i=1}^m \zeta_l^{r_i c_i}=1$. If this does not hold then it is easy to see from Theorem 3 that for each prime $\mathbf{p}$ there is an $i$ such that $(\frac{a_i}{\mathbf{p}})_l\neq \zeta_l^{r_i}$. Thus if $r_i$'s satisfy the required condition then $S_x$ exactly counts number of primes of Norm up to $x$ which satisfy $(\frac{a_i}{\mathbf{p}})_l=\zeta_l^{r_i}$ for all $i$. In case there is inconsistency among choices of $r_i$, then $S_x=0$. \\

Now to estimate $S_x$, we can actually pass down to the corresponding counting funtion for $b_j^{'}s$ which also will be denoted by $S_x$. When we change from the set $S=\{a_1, \ldots ,a_m\}$ to the set $T=\{b_1, \ldots ,b_t\}$ obtained as in the algorithm in section 2, then, the given $m$ elements $\zeta_l^{r_i}$ uniquely determine a set of $t$ elements $\zeta_l^{s_j}$ such that $(\frac{a_i}{\mathbf{p}})_l=\zeta_l^{r_i}, \mbox{ for all } 1 \leq i \leq m$ iff $(\frac{b_j}{\mathbf{p}})_l=\zeta_l^{s_j} \mbox{ for all } 1 \leq j \leq t$. If the conditions $ (\frac{a_i}{\mathbf{p}})_l=\zeta_l^{r_i}, \mbox{ for all }1 \leq i \leq m$ lead to a condition of the form $(\frac{b_j}{\mathbf{p}})_l=\zeta_l^{s_j} \mbox{ for all } 1 \leq j \leq t$ with $b_j$ an $l^{th}$ power and $s_j\neq 0$, then we can immediately conclude that there is no prime $\mathbf{p}$ satisfying the condition. Studying the counting function with the $b_j$ makes it easier, since there will be only one main term with one root of unity in it (not a sum of roots of unity). Hence, it is enough to study the behaviour of primes $\mathbf{p}$ which satisfy $(\frac{b_j}{\mathbf{p}})_l=\zeta_l^{s_j} \mbox{ for all } 1 \leq j \leq t$. Now we consider the counting function\\
$$S_x=\frac{1}{vl^t} \sum_{\mathbf{p}; Norm(\mathbf{p}) \leq x, \mathbf{p} \notin S^{'}} \prod_{k=1}^t \left(\prod_{j=1, j\neq r_k}^l(\zeta_l^j-(\frac{b_k}{\mathbf{p}})_l)\right).$$
Here $S^{'}$ is the set of primes dividing $b_1 \ldots b_t$ and $v$ is a unit satisfying $$vl^t=\prod_{k=1}^t \prod_{j=1, j\neq s_k}^l(\zeta_l^j-\zeta_l^{s_k}).$$ We emphasize that $S_x$ exactly counts number of primes of Norm up to $x$ which satisfy $(\frac{a_i}{\mathbf{p}})_l=\zeta_l^{r_i}$ for all $i$. Because of multiplicativity of $l^{th}$ power residue symbol one obtains $$S_x=\frac{1}{u l^t} \sum_{\mathbf{p}; Norm(\mathbf{p}) \leq x, \mathbf{p}\notin S^{`}} \sum_{0\leq d_i \leq l-1, n=\prod b_i^{d_i}} \zeta_l^{t_n}(\frac{n}{\mathbf{p}})_l \mbox{ for some integer }t_n.$$ After changing the order of summation, we see that here main term comes from those $n$ which are $l^{th}$ power, by Theorem 6, and if $n$ is not an $l^{th}$ power then the contribution is $o(\pi(x))$. But from the construction of $b_j$ its clear that no $n \neq 1$ will be an $l^{th}$ power. Hence, the main term will give, in absolute value, $\frac{1}{ l^t}(\pi(x)-|S^{'}|)$. Thus density of the primes $\mathbf{p}$ satisfying $(\frac{b_j}{\mathbf{p}})_l=\zeta_l^{s_j} \mbox{ for all } 1 \leq j \leq t$ and hence satisfying $(\frac{a_i}{\mathbf{p}})_l=\zeta_l^{r_i}, \mbox{ for all } 1 \leq i \leq m$ is $\frac{1}{l^t}$.

\begin{rem}
1. Note that the density does not depend upon the choice of $r_i$ as long as long as there is consistency required.\\
2. One can also obtain the density of primes $\mathbf{p}$ of $\mathbb{Q}(\zeta_l)$ which satisfy $(\frac{\alpha_i}{\mathbf{p}})_l =\zeta_l^{r_i}$, where $\alpha_i$ are integers of $\mathbb{Q}(\zeta_l)$ and $r_i$'s are as in the Introduction. The above proof may not work in this case. However we see from the second definition of $l^{th}$ residue symbol in section 3 one immediately notices that if the requirement $\left(\frac{\alpha_i}{\mathbf{p}}\right)_l=\zeta_l^{r_i}$ is consistent (in the same sense as in section 3) then it uniquely determines an element in $Gal(\mathbb{Q}(\zeta_l,\alpha_1^{\frac{1}{l}}, \ldots ,\alpha_m^{\frac{1}{l}})/ \mathbb{Q}(\zeta_l))$, and hence density of such primes $\mathbf{p}$ is $[\mathbb{Q}(\zeta_l,\alpha_1^{\frac{1}{l}}, \ldots ,\alpha_m^{\frac{1}{l}}): \mathbb{Q}(\zeta_l)]^{-1}$.
\end{rem}

\section{Another way to find the degree}
Let $p_1,p_2, \ldots ,p_n$ be all the primes dividing $a_1 \ldots a_m$. Let us write $\lambda_{ij}$ for exact power of $p_j$ dividing $a_i$. Then we will consider the $m \times n$ matrix $T$ whose $(i,j)^{th}$ entry is $\lambda_{ij}$. Note that for our purpose we can assume that $0 \leq \lambda_{ij} \leq l-1$.

\begin{lem}
 The cardinality of the set $A=\{(\lambda_i)_{i=1}^m : 0\leq \lambda_i \leq l-1, \prod_i a_i^{\lambda_i} \in \mathbb{Z}^l\}$ is a power of $l$.
\end{lem}

\begin{proof}
 Consider the $\mathbb{Z}/l \mathbb{Z}$ vector space $(\mathbb{Z}/l \mathbb{Z})^m$ with basis $S=\{a_1, \ldots ,a_m\}$. $\mathbb{Z}/ l\mathbb{Z}$ acts on $\mathbb{Q}^*/ (\mathbb{Q}^*)^l \mbox{ by } \alpha . x=x^{\alpha}$. Consider the map T:$(\mathbb{Z}/ l \mathbb{Z})^m \longrightarrow \mathbb{Q}^*/ (\mathbb{Q}^*)^l$ which sends $a_i\rightarrow \bar{a_i}$ and extend it linearly then $\sum_i \lambda_i a_i \in kerT \mbox{ iff } (\lambda_i)\in A$. This proves that $|A|$ is an $l^{th}$ power.
\end{proof}

As mentioned in \cite{SHW} the degree $\mathbb{Q}(a_1^{\frac{1}{l}}, \ldots ,a_m^{\frac{1}{l}},\zeta_l)/\mathbb{Q}(\zeta_l)$ is $l^m/l^r$ where $l^r$ is the cardinality of set $A$. Now we relate this degree to the rank of the matrix $T$.

\begin{thm}
 The rank of the matrix $T$ is $m-r$.
\end{thm}

\begin{proof}
 If there are $x_i, 1\leq i \leq m \mbox{ with } 0\leq x_i \leq l-1$ such that $\prod_{i=1}^m a_i^{x_i} \in \mathbb{Z}^l$, then for all $j$ we have $x_1 \lambda_{1j}+ \ldots +x_m \lambda_{mj}=0 (\mbox{ mod }l)$, i.e the row vectors $(\lambda_{i1},  \ldots , \lambda_{in}), 1 \leq i \leq m$ in $(\mathbb{Z}/{l \mathbb{Z}})^n$ are linearly dependent.
Conversely any such linear dependence among row vectors of matrix $T$ will give exactly $l$ many relations of the type in set $A$. Let $r$ be the rank of the matrix $T$. After a rearrangement we can assume that the row vectors $(\lambda_{i1},  \ldots , \lambda_{in}), 1 \leq i \leq r$ are linearly independent. Then we see that for any choice of $x_i, 1\leq i \leq r \mbox{ with } 1\leq x_i \leq l-1$ the condition $\prod_{i=1}^r a_i^{x_i} \in \mathbb{Z}^l$ does not hold. On the other hand for any selection of $x_j, j>r$ we have $l$ many relation of the form $\prod_{i=1}^m a_i^{x_i} \in \mathbb{Z}^l$ (this can be seen by looking at the vectors $(\lambda_{i1},  \ldots , \lambda_{in}), 1 \leq i \leq r$  and $x_{r+1}(\lambda_{r+11},  \ldots , \lambda_{r+1n})+ \ldots +x_m(\lambda_{m1},  \ldots , \lambda_{mn})$ which are linearly dependent). This proves the theorem.
\end{proof}

\noindent{\bf Acknowledgement.}
We thank Professor Kannan Soundararajan for pointing out some mathematical inaccuracies in an earlier draft. Thanks are also due to the anonymous referee for a careful reading of the manuscript and helping us in improving the presentation.


\begin{thebibliography}{1}
 \bibitem{RB}
R. Balasubramanian, F. Luca, R. Thangadurai, On the exact degree of $\mathbb{Q}(a_1^{\frac{1}{l}}, \ldots ,a_m^{\frac{1}{l}})$ over $\mathbb{Q}$, Proceedings of the American Mathematical Society, Volume 138, number 7, July 2010, pages 2283-2288.

\bibitem{SHW}
Steven H. Weintraub, Galois Theory, Springer-Verlag 2006 (Universitext)

\bibitem{CF}
Cassels J. W. S., Frohlich, A., Algebraic Number Theory, The London Mathematical Society, 1967.

\bibitem{JN}
Jurgen Neukirch, Algebraic Number Theory, Springer 1991.

\bibitem{HK}
Helmut Koch, Number Theory: Algebraic Numbers and Functions, Graduate studies in mathematics, volume 24, 2000.

\bibitem{LW}
Lawrence C. Washington, Introduction to Cyclotomic Fields, Second Edition, Springer 1991.

\bibitem{DH}
David Hilbert, The theory of Algebraic Number Fields, 1991.

\bibitem{EM}
M. Ram Murty and Jody Esmonde, Problems in Algebraic Number Theory, Graduate Text in Mathematics, 1991.
\end{thebibliography}
\end{document}